\documentclass[a4paper,12pt]{amsart}
\usepackage[T1]{fontenc}
\usepackage[english]{babel}
\usepackage{amsthm, amssymb, amsfonts}
\usepackage{dsfont}
\usepackage{mathtools}
\usepackage{mathrsfs}
\usepackage[margin=1in]{geometry}

\usepackage{biblatex}
\bibliography{references}
\usepackage{csquotes}
\usepackage{amsaddr}
\usepackage{esint}
\numberwithin{equation}{section}
\theoremstyle{plain}
\newtheorem{theorem}{Theorem}[section]
\newtheorem{proposition}[theorem]{Proposition}

\newtheorem{lemma}[theorem]{Lemma}
\theoremstyle{remark}
\newtheorem{remark}[theorem]{Remark}

\theoremstyle{definition}
\newtheorem{definition}[theorem]{Definition}

\DeclarePairedDelimiter{\inn}{\langle}{\rangle}
\DeclarePairedDelimiter{\abs}{\lvert}{\rvert}
\DeclarePairedDelimiter{\norm}{\lVert}{\rVert}
\DeclarePairedDelimiter{\seq}{(}{)}
\newcommand{\PER}{\operatorname{P}}
\newcommand{\VOL}{\operatorname{V}}

\newcommand{\F}{\mathcal{F}}

\newcommand{\Per}{\operatorname{Per}_g}
\newcommand{\Vol}{\operatorname{Vol}_g}

\newcommand{\dist}{\operatorname{dist}_g}
\newcommand{\bary}{\operatorname{Bar}_g}

\newcommand{\vol}{\operatorname{vol}}
\newcommand{\argmin}{\operatorname{arg min}}

\newcommand{\Ric}{\operatorname{Ric}}

\newcommand{\diva}{\operatorname{div}}
\newcommand{\Haus}{\mathscr{H}}

\newcommand{\Ham}{\mathbb{H}}
\newcommand{\R}{\mathbb{R}}
\newcommand{\C}{\mathbb{C}}
\newcommand{\N}{\mathbb{N}}
\newcommand{\Z}{\mathbb{Z}}
\newcommand{\Oct}{\mathbb{O}}

\newcommand{\K}{\mathbb{K}}

\newcommand{\Sp}{{\mathbf{S}^{n-1}}}
\newcommand{\B}{\mathbf{B}^n}
\newcommand{\Ri}{\mathbf{R}}
\newcommand{\om}{\boldsymbol\omega}
\newcommand{\brho}{\boldsymbol\rho}

\newcommand{\Dist}{\mathcal{H}}
\newcommand{\Vist}{\mathcal{V}}

\newcommand{\spt}{\operatorname{spt}}

\newcommand{\diam}{\operatorname{diam}}

\title[Quantitative $C^1$-stability of spheres in rank one symmetric spaces]{Quantitative $C^1$-stability of spheres in rank one symmetric spaces of non-compact type}
\author{Lauro Silini}
\address{Department of Mathematics, ETH Zürich, Zürich, Switzerland}
\email{lauro.silini@math.ethz.ch}
\begin{document}
\maketitle
\begin{abstract}
We prove that in any rank one symmetric space of non-compact type $M\in\{\R H^n,\C H^m,\Ham H^m,\Oct H^2\}$, geodesic spheres are uniformly quantitatively stable with respect to small $C^1$-volume preserving perturbations. We quantify the gain of perimeter in terms of the $W^{1,2}$-norm of the perturbation, taking advantage of the explicit spectral gap of the Laplacian on geodesic spheres in $M$. As a consequence, we give a quantitative proof that for small volumes, geodesic spheres are isoperimetric regions among all sets of finite perimeter.
\end{abstract}
\section{Introduction}
\subsection{Background and motivations}
In a given Riemannian manifold $(M,g)$ an isoperimetric region of volume $0<v<\Vol(M)$ is a compact, smooth subset $E\subset M$ attaining the isoperimetric profile
\[
\Per(E)=I_M(v):=\inf\{\Per(F):F\subset M,\Vol(F)=v\},
\]
where $\Per(\cdot)$ denotes the Riemannian perimeter and $\Vol(\cdot)$ the Riemannian volume.
We consider $M=\K H^m$ to be a rank one symmetric space of non-compact type, that is one among the real $\R H^m$, complex $\C H^m$ and quaternionic $\Ham H^m$ hyperbolic spaces and the Cayley plane $\Oct H^2$. They share a number of very fine properties, in particular: up to renormalization their sectional curvature lies in $[-4,-1]$. Moreover, they are symmetric in the sense that the curvature tensor is invariant under parallel transport, or said otherwise, every geodesic symmetry can be extended to a global isometry. This is why they can arguably be considered as the natural generalization of space forms in the context of non-constant curvature manifolds. In particular, the fact that the isotropy group acts transitively on geodesic spheres, together with the concept of classic symmetrization procedures applied to constant curvature manifolds (see Steiner \cite{Steiner} and Schwarz \cite{Schwarz}) lead to the natural conjecture mentioned by Gromov \cite[6.28$\frac12_+$]{Gromov} and Ros that geodesic balls are isoperimetric regions in this class of manifolds. This was proved to be true only in the constant curvature case $M=\R H^m$.
\par Ensured by the cocompact action of the isometry group on $M$, existence of isoperimetric sets for all volumes is a classic fact: see Morgan \cite{FM1} in the context of bubble sets, and the very general formalization in the article of Galli and Ritor\'e \cite{GR}. It is also known that spheres are stable under infinitesimal volume preserving perturbations, as it was shown for $\C H^m$ by Barbosa, do Carmo and Eschenberg in their celebrated article \cite{BCE} (see \cite{RT} for the general case). Their proof is a natural consequence of the following variational principle: let $\Sigma$ be an oriented, closed and immersed submanifold of $M$, and let $\Ric_\Sigma$ and $\mathrm{I\!I}_\Sigma$ be the induced Ricci curvature tensor and second fundamental form of $\Sigma$. If
\begin{equation}\label{eq:BDC_stab}
\Ric_\Sigma+\norm{\mathrm{I\!I}_\Sigma}^2=\text{ constant }=\lambda,
\end{equation}
then $\Sigma$ is stable if and only if $\lambda=\lambda_1$, the first eigenvector of the induced Laplacian $\Delta_\Sigma$. Thanks to a peculiar characterization of the Laplace spectrum over fibre bundles with totally geodesic immersed fibres (about this the work by Bergery and Bourguignon \cite{BB}), the first eigenvalue of the Laplacian on geodesic spheres is explicitly computable, reducing the proof of stability to a direct check of Equation \eqref{eq:BDC_stab}. 
\par
As a natural question one can ask if it would be possible to improve this result by quantifying the loss of optimality in terms of the size of the perturbation. In $\R^n$ the first result in this direction is due to Fuglede \cite{Fuglede}. He showed that for every Lipschitz volume preserving perturbation of the unit ball of the form
\[
\partial E:=\{(1+\rho(\varphi))\varphi:\varphi\in S^{n-1}\},
\]
with barycenter at the origin, the isoperimetric deficit $\PER(E)-\PER(B(0,1))$ can be estimate from below by
\begin{equation}\label{eq:ref:F}
\PER(E)-\PER(B(0,1))\geq\frac12\int_{S^{n-1}}\abs{\nabla \rho}^2\,dx-(n-1)\int_{S^{n-1}}\rho^2\,dx+o\bigl(\norm{\rho}_{W^{1,2}(S^{n-1})}\bigr),
\end{equation}
taking advantage of the integral representation of the perimeter and the volume of $E$ in terms of $\rho$. Developing $\rho$ over spherical harmonics, the first eigenvalue is exactly equal to $(n-1)$, compensating the negative term in Equation \eqref{eq:ref:F}, and finally proving the existence of a constant $C=C(n)>0$ such that
\[
\PER(E)-\PER(B(0,1)))\geq C\norm{\rho}_{W^{1,2}(S^{n-1})}^2.
\]
\par As mentioned before, the information we have about the spectral decomposition of the Laplace operator over geodesic spheres in $M$ coming from the intrinsic relation with the Hopf fibration
\[
S^{d-1}\to S^{n-1}(r)\to \K P^{m-1},
\]
suggests that the method of Fuglede might be extended also to all rank one symmetric spaces of non-compact type, and this is precisely what we are going to prove in the first part of this paper. 
\par In the second part, we address our attention to what happens in the small volume regime. It is known that for all Riemannian manifolds with cocompact isometry group, all isoperimetric regions with sufficiently small volume are invariant under the action of the stabilizer of their center of mass. This result was mentioned first by Tomter in the context of the Heisenberg group in \cite{TOM}, referring to an unpublished article by Kleiner. Later, we can find the complete proof as a corollary of a more general result in the article by Nardulli introducing the concept of pseudo-bubbles, see \cite{NARDU}. Since spheres are the only surfaces preserved by the action of the isotropy group, we have as a direct implication that they are the unique isoperimetric regions in rank one symmetric spaces in the small volume regime.
\par The proof of the before mentioned theorem relies on an implicit argument. We give a new quantitative proof by taking advantage of the strong stability of spheres \`a la Fuglede. To do so, we make use of the deep quantitative stability estimates in $\R^n$ in the context of finite perimeter sets
\[
\PER(E)-\PER(B(0,1))\geq C(n)\alpha(E)^2,
\]
where
\[
\alpha(E):=\min_{y\in\R^n}\Bigl\{\VOL(E\triangle B(y,1)):\VOL(B(y,1))=\VOL(E)\Bigr\},
\]
is the Fraenkel asymmetry index and $E\triangle B(y,1)$ stands for the symmetric difference. The proof of this result was first given by Fusco Maggi and Pratelli in \cite{fusco2008sharp}, and then further simplified by Cicalese and Leonardi in \cite{cicalese2012selection}. The same result was then extended to anisotropic perimeter functionals by Figalli, Maggi and Pratelli by mean of optimal transportation techniques in \cite{FigalliMaggi}, further improved in the Euclidean setting by Fusco and Julin in \cite{Fusco}, and extended to the real hyperbolic space by B\"ogelein, Duzaar and Scheven in \cite{duzaar}. We show that rescaled optimal sets are almost-minimizing with respect to the Euclidean metric on the lifted tangent space, that is they are isoperimetric up to an error uniformly proportional to the size of the perturbation.
Then, we take advantage of the quantitative stability results in $\R^n$, to prove $L^1$-proximity with geodesic spheres that can be improved, again by almost-minimality, to $L^\infty$-proximity, provided that we choose the barycenter of the shape as a lifting point. The desired regularity follows by the regularity theory developed by Figalli in \cite{figalli2017regularity}.
\subsection{Main results}
The goal of this paper is to show the following results.
\begin{theorem}\label{thm:2}Let $M=\K H^m$ be a rank one symmetric space of non-compact type and $R_0>0$ any fixed radius. Let $E\subset M$ be a volume preserving perturbation of $\B(R)$, $0<R\leq R_0$, with boundary of the form
\[
\partial E=\{\exp_o(R(1+\rho(\varphi))):\varphi\in S^{n-1}\},
\]
where $o\in M$ is a fixed base-point, and $\rho\in C^1(S^{n-1},(-1,+\infty))$. Denote with $\brho:\Sp(R)\to(-1,+\infty)$ the perturbation $\rho$ viewed as a function from the geodesic sphere in $M$, given in normal coordinates as
\[
\brho(R\varphi)=\rho(\varphi), \quad \varphi\in S^{n-1}.
\]
Then, there exist $\varepsilon=\varepsilon(M,R_0)>0$ and $C=C(M,R_0,R)>0$, such that
\begin{align*}
\Per(E)-\Per(\B(R))&\geq C\bigl(\norm{\brho}_{L^2(\Sp(R))}^2+\norm{\nabla^g\brho}_{L^2(\Sp(R))}^2\bigr),
\end{align*}
provided $\norm{\rho}_{C^1}\leq \varepsilon$. In particular, if $E$ is isoperimetric, then $E=\B(R)$.
\end{theorem}
To establish this result, we will demonstrate the following explicit lower bound under the technical assumption that $\rho$ is barycentric preserving. To obtain Theorem \ref{thm:2} we can compensate for this assumption with a small transvection (that amounts to a translation obtained via a composition of central symmetries) of the perturbed set.
\begin{theorem}\label{thm:1} Under the same assumptions of Theorem \ref{thm:2}, suppose additionally that $E$ has barycenter in $o\in M$.
Denote with $\lambda_2^R$ the second eigenvalue of the Laplacian over $\Sp(R)$. Then, there exists $\varepsilon=\varepsilon(M,R_0)>0$ such that
\begin{align*}
\Per(E)-\Per(\B(R))&\geq\frac{R^2\lambda_2^R}{48}\norm{\brho}_{L^2(\Sp(R))}^2+\frac{R^2}{32}\norm{\nabla^g\brho}_{L^2(\Sp(R))}^2,
\end{align*}
provided $\norm{\rho}_{C^1}\leq \varepsilon$.
\end{theorem}
As an application of Theorem \ref{thm:1}, we will give a new quantitative proof of the isoperimetric problem in the small volume regime.
\begin{theorem}\label{thm:small_volumes}Let $M=\K H^m$ be a rank one symmetric space of non-compact type. Then, there exists a possibly computable volume $\bar v=\bar v(M)>0$ such that all geodesic balls $\B(R)\subset M$ with volume $\Vol(B(R))<\bar v$ are uniquely isoperimetric in $M$.
\end{theorem}

\begin{remark}With exactly the same arguments, the results of Theorem \ref{thm:2} and Theorem \ref{thm:1} hold true under the weaker assumption of $\rho$ belonging to the Sobolev space $W^{1,\infty}(S^1,(-1,+\infty))$.
\end{remark}
\section{Preliminaries}\label{sec:preliminaries}
For every real, finite dimensional division algebra $\K\in\{\R,\C,\Ham,\Oct\}$, we consider $(M,g)=(\K H^m,g)$ to be the associated rank one symmetric space of non-compact type with complex dimension $\dim_\K(M)=m$, and real dimension $\dim_{\R}(M)=md$, where $d=\dim(\K)\in\{1,2,4,8\}$. In the octonionic case, there is only the 16-dimensional Cayley plane $\Oct H^2$. Let $\dist(\cdot,\cdot)$ be the induced Riemannian distance. As classic references on symmetric spaces we cite the books of Eberlein \cite{Eberlein} and Helgason \cite{Helgason}. The Riemannian manifold $M$ can be realized as the quotient of a symmetric pair $(G,K)$, where $G$ is a semisimple Lie group acting transitively on $M$, and $K$ represents the isotropy group, that is defined as all elements in $G$ fixing an arbitrarily chosen point in $M$. In particular
\begin{center}
\begin{tabular}{ c|c| c } 
$M=G/K$ & $G$ & $K$ \\
\hline
$\R H^{m}$ & $\operatorname{SO}(m,1)$ & $\operatorname{SO}(m)$ \\ 
$\C H^m$&   $\operatorname{SU}(m,1)$ & $\operatorname{S}(\operatorname{U}(m)\operatorname{U}(1))$\\ 
$\Ham H^m$ & $\operatorname{Sp}(m,1)$ & $\operatorname{Sp}(m)\operatorname{Sp}(1)$ \\ 
$\Oct H^2$ & $\operatorname{F}^{-20}_4$ & $\operatorname{Spin}(9)$
\end{tabular}
\end{center}
The adjective \emph{symmetric} comes from the fact that the Riemannian curvature tensor $\Ri(X,Y)Z$ is parallel along geodesics. Equivalently, all geodesic reflections can be extended to a global isometry. The \emph{rank} measures the dimension of a maximal isometrically embedded flat. In our case it means that geodesics are the only flat submanifolds in $M$. When $\K=\R$, $M$ is the real hyperbolic space of constant negative curvature. In general, up to renormalization $M$ has sectional curvature lying in $[-4,-1]$ and it is simply connected, hence diffeomorphic to $\R^n$. Moreover, $M$ is two-points homogeneous, that is every couple of points can be mapped by an isometry to any other couple of points with the same mutual distance. It follows that geodesic spheres
\[
\Sp(o,R):=\{x\in M:\dist(x,o)=R\},
\]
centered at $o\in M$ with radius $R>0$, are homogeneous submanifolds of constant mean curvature. Since all spheres with the same radius are isometric, we will often denote with $\Sp(R)$ a generic geodesic sphere in $M$ of radius $R>0$ with induced metric that we will keep calling $g$. Analogously, we denote with 
\[
\B(o,R):=\{x\in M:\dist(x,o)<R\},
\]
the open geodesic ball centered in $o\in M$ with radius $R>0$, and with $\B(R)$ a generic open geodesic ball of radius $R>0$ in $M$.
\subsection{Distributions and spectral decomposition on spheres}\label{sec:distribution_spectral_decomposition}
For every non-zero vector $N_x$ at $x\in M$, the Jacobi operator
\[
\Ri(\cdot,N_x)N_x:T_xM\to T_xM,
\]
has exactly three eigenvalues: $\{-4,-1,0\}$. Denoting with $\Dist_x$ and $\Vist_x$ the eigenspaces associated to the eigenvalues $-4$ and $-1$ respectively, the tangent plane $T_xM$ splits orthogonally as
\begin{equation}\label{eq:splitting}
T_xM=\Dist_x\oplus\Vist_x\oplus\R N_x,
\end{equation}
where $\dim_\R(\Dist_x)=(d-1)$, and $\dim_\R(\Vist_x)=d(n-1)$. Hence, for every non-vanishing vector field $N\in\Gamma(TU)$ defined on an open set $U\subset M$, the maps $x\mapsto\Dist_x$ and $x\mapsto\Vist_x$ induce two well defined distributions $\Dist$ and $\Vist$ on $U$. We will denote the orthogonal projections with
\begin{align*}
(\cdot)^h:&TU\to\Dist,\\
(\cdot)^v:&TU\to\Vist,\\
(\cdot)^n:&TU\to \R N.
\end{align*}
Notice that when $\K=\R$, then $\Dist=\emptyset$. In this exceptional case we set $(\cdot)^v\equiv 0$. In particular, when $N$ is a radial vector field emanating from a base point $o\in M$, then the orthogonality of \eqref{eq:splitting} implies that the tangent bundle of any sphere $\Sp(o,R)$ splits orthogonally with respect to $g$ as the direct sum of $\Dist$ and $\Vist$ restricted to $\Sp(o,R)$. Turns out that this splitting also arises from the vertical and horizontal distribution associated to the celebrated Hopf fibration of Euclidean spheres
\[
S^{d-1}\to S^{m-1}(R)\to \K P^{m-1},
\]
where $\K P^{m-1}$ is the real, complex, quaternionic and octonionic projective space of complex dimension $\dim_\K(\K P^{m-1})=m-1$, respectively. This particular structure allows computations about the spectral decomposition of $L^2(\Sp(R),g)$ with respect to the induced Riemannian Laplacian, see \cite{BCE,BB,RT} and very recently \cite{bettiol2022laplace}. In our case, it will be sufficient to know that the associated eigenvalues $\{\lambda_i^R\}_{i\in\N}$ satisfy the bound
\[
\lambda_j^R\geq\frac{j(j+d-2)+j(n-d)\cosh^2(R)}{\sinh^2(R)\cosh^2(R)},
\]
with equality when $j=1$. We will denote with 
\[
\{f_{j,k}^R\in L^2(\Sp(R),g):1\leq k\leq n_j\}_{j\geq 0}.
\]
the spherical harmonics with muliplicity $n_j\geq 1$ constituting an orthogonal basis of $L^2(\Sp(R),g)$, so that
\[
\norm{\nabla^{g}f^R_{j,k}}_{L^2(\Sp(R),g)}^2=\lambda_j^R \norm{f^R_{j,k}}_{L^2(\Sp(R),g)}^2,
\]
where $\nabla^{g}$ denotes the Riemannian gradient with respect to $g$ on $\Sp(R)$.
\subsection{Useful geometric identities by comparison}
Denote with $g_e=\inn{\cdot,\cdot}$ and $\abs{\cdot}$ the usual Euclidean metric and norm on $\R^n$, and with $S^{n-1}(x,r)$ and $B^{n}(x,r)$ the Euclidean spheres and open balls centered in $x\in\R^n$ with radius $r>0$. As usual, $S^{n-1}$ and $B^n$ denote generic unit spheres and balls. In order to do computations in $(M,g)$ we decided to work in normal coordinates $(r,\varphi)\in(0,+\infty)\times S^{n-1}$. Let $\PER(\cdot)$ and $\VOL(\cdot)$ be the perimeter and volume operators in $\R^n$ with respect to the Euclidean metric. Set
\[
\om_n:=\VOL(B^{n}).
\]
From now on, $o\in M$ will be an arbitrarily fixed base-point, if not stated otherwise. Taking the pullback metric $\exp_o^*g$ we can identify isometrically $M$ with $\R^n$. Chose now and for all $N\in \Gamma(M\setminus\{o\})$ to be the radial, unit vector field emanating from $o$. Thanks to the previous discussion, we can find an explicit formula relating $g_e$ with $g$.
\begin{lemma}\label{lem:comparing_manifolds}For every $x=(r,\varphi)\in M\setminus\{o\}$, the splitting
\[
T_xM=\Dist_x\oplus\Vist_x\oplus \R N_x,
\]
is orthogonal with respect to $g_e$. In particular, one has that
\begin{equation}\label{eq:relation}
g(X,Y)=\inn{X^n,Y^n}+\frac{\cosh^2(r)\sinh^2(r)}{r^2}\inn{X^h,Y^h}+\frac{\sinh^2(r)}{r^2}\inn{X^v,Y^v},
\end{equation}
for all $X,Y\in T_xM$.
\end{lemma}
\begin{proof}Fix an arbitrary unit direction $N_o\in T_oM$, and let $V_o\in T_oM$ be any vector orthogonal to it with respect to $g\vert_o=g_e\vert_o$. Since the radial geodesics emanating from $o$ with respect to $g$ are the same as the Euclidean ones, the Jacobi field $Y(t)$ along the geodesic $\sigma:t\mapsto tN_o$, determined by the initial conditions $Y(0)=0$, $\dot Y(0)=V_o$ is the same for both metrics. Let $V(t)$ and $V_e(t)$ be the parallel transport of $V_o$ along $\sigma$ with respect to $g$ and $g_e$, respectively. By the very definition of symmetric spaces, the curvature tensor $R$ is itself parallel along geodesics. This implies that
\[
tV_e(t)=Y(t)=\frac{\sinh(\sqrt{-\kappa}t)}{\sqrt{-\kappa}}V(t),
\]
provided $V_o$ belongs to the $\kappa$-eigenspace of the Jacobi operator $R(\cdot,N_o)N_o$. Therefore, parallel vector fields in the eigenspaces are collinear for the two metrics. Hence, for $t>0$ the linear subspaces $\Dist_{\sigma(t)}$ and $\Vist_{\sigma(t)}$ are nothing else than the parallel transport of the corresponding eigenspaces of $R(\cdot,N_o)N_o$ along $\sigma$. It follows that the splitting $T_xM=\Dist_x\oplus \Vist_x$ is orthogonal not only with respect to $g$, but also with respect to the Euclidean metric $g_e$. Equation \eqref{eq:relation} is a direct consequence of this fact and the definition of the distribution $\Dist$ and $\Vist$.
\end{proof}
In particular, the volume density on $M$ associated to the metric $g$ is radial and given by
\[
\omega_g(r,\varphi)=\omega_g(r)=\frac{\sinh^{n-1}(r)\cosh^{d-1}(r)}{r^{n-1}}.
\]
Let $\Per(\cdot)$ and $\Vol(\cdot)$ be the perimeter and volume operators in $M$. As a consequence of the previous Lemma we have the following formulae.
\begin{lemma}\label{lem:formulae_vol_per}Let $E$ be a subset of $M$ with smooth boundary. Then, in normal coordinates  we have that
\begin{equation}\label{eq:volume}
\Vol(E)=\int_E\omega_g(r)\,d\Haus^n,
\end{equation}
and
\begin{equation}\label{eq:perimeter}
\Per(E)=\int_{\partial E}\omega_g(r)\Bigl(\abs{\nu^n}^2+\frac{r^2}{\cosh^2(r)\sinh^2(r)}\abs{\nu^h}^2+\frac{r^2}{\sinh^2(r)}\abs{\nu^v}^2\Bigr)^{1/2}\,d\Haus^{n-1},
\end{equation}
where $\nu$ denotes the normal vector field to $\partial E$ with respect to $g_e$.
\end{lemma}
\begin{proof}Equation \eqref{eq:volume} is tautological. We prove Equation \eqref{eq:perimeter}. Denoting with 
\[
\vol_g:=\omega_g(r)dx=\omega_g(r)dx^1\wedge\dots\wedge dx^n,
\]
the volume form in $M$, and with $\nu_g$ and $\nu$ the normal vector field of $\partial E$ with respect to $g$ and $g_e$ respectively, we have that
\[
\Per(E)=\int_{\partial E}\iota_{\nu_g}\vol_g=\int_{\partial E}\omega_g(r)\iota_{\nu_g}\vol,
\]
where we denote the interior product $(\iota_{\nu_g}\vol_g)(\cdot)=\vol_g(\nu_g,\cdot)\in\Omega^{n-1}(\partial E)$. Now, for a fixed $x\in\partial E$, choose an orthonormal basis $\{v_2,\dots,v_n\}$ of $T_x\partial E$ orthogonal to $\nu$ with respect to $g_e$. Then,
\begin{align*}
(\iota_{\nu_g}\vol)_x(v_2,\dots,v_n)&=\vol_x(\nu_g,v_2,\dots,v_n)=\inn{\nu_g,\nu}\vol_x(\nu,v_2,\dots,v_n)\\
&=\inn{\nu,\nu_g}(\iota_\nu\vol)_x(v_2,\dots,v_n),
\end{align*}
showing that
\[
\Per(E)=\int_{\partial E}\omega_g(r)\inn{\nu,\nu_g}\iota_\nu\vol=\int_{\partial E}\omega_g(r)\inn{\nu,\nu_g}\,d\Haus^{n-1}.
\]
We are left to compute $\inn{\nu,\nu_g}$. By Lemma \ref{lem:comparing_manifolds} we have that
\[
\tilde\nu_g:=\nu^n+\frac{r^2}{\cosh^2(r)\sinh^2(r)}\nu^h+\frac{r^2}{\sinh^2(r)}\nu^v,
\]
realizes $g(\tilde\nu_g,v_i)=\inn{\nu,v_i}=0$ for all $i=2,\dots,n$, implying that $\tilde\nu_g$ is collinear to $\nu_g$. Since $g(\tilde\nu_g,\tilde\nu_g)=\inn{\nu,\tilde\nu_g}$ we get that
\begin{align*}
\inn{\nu,\nu_g}&=\inn{\nu,\tilde\nu_g}g(\tilde\nu_g,\tilde\nu_g)^{-1/2}=\inn{\nu,\tilde\nu_g}^{1/2}\\
&=\Bigl(\abs{\nu^n}^2+\frac{r^2}{\cosh^2(r)\sinh^2(r)}\abs{\nu^h}^2+\frac{r^2}{\sinh^2(r)}\abs{\nu^v}^2\Bigr)^{1/2},
\end{align*}
concluding the proof of the lemma.
\end{proof}
Define the barycenter of $E$ as
\[
\bary(E):=\argmin_{p\in M}\Bigl\{\int_E \dist^2(x,p)\,d\vol_g(x)\Bigr\}\in M.
\]
It is always unique and well defined since the negative curvature of $M$ implies that the above functional is strictly convex in $p\in M$, see \cite[Section 2.5]{duzaar}. Differentiating, we have that $p=\bary(E)$ if and only if
\begin{equation}\label{eq:barycenter_exp}
-2\int_E\exp^{-1}_p(x)\,d\vol_g(x)=0.
\end{equation}
In the  normal coordinates pointed at $\bary(E)$, this reads as
\begin{equation}\label{eq:barycenter_at_the_origin}
0=\int_Ex\,d\vol_g(x)=\int_E r\varphi\omega_g(r)\,d\Haus^n.
\end{equation}
In the particular case in which $\partial E$ is a $C^1$-radial perturbation of $\Sp(R)$
\begin{equation}\label{eq:starshaped+boundary}
\partial E=\{\exp_o(R(1+\rho(\varphi))):\varphi\in S^{n-1}\}=\{(R(1+\rho(\varphi)),\varphi):\varphi\in S^{n-1}\},
\end{equation}
for some $C^1$-function $\rho:S^{n-1}\to(-1,+\infty)$, then the normal $\nu$ with respect to $g_e$ is given by
\[
\nu=\Bigl(\varphi-\frac{\nabla\rho}{1+\rho}\Bigr)\Bigl(1+\frac{\abs{\nabla\rho}^2}{(1+\rho)^2}\Bigr)^{-1/2},
\]
where $\nabla$ denotes the gradient with respect to the round metric on $S^{n-1}$. Applying Equation \eqref{eq:perimeter} of Lemma \ref{lem:formulae_vol_per} one gets that
\begin{equation}\label{eq:starshaped_per}
\Per(E)=\int_{S^{n-1}}\omega_g(r)r^{n-1}\Bigl(1+R^2\frac{\abs{\nabla^h\rho}^2+\cosh^2(r)\abs{\nabla^v\rho}^2}{\sinh^2(r)\cosh^2(r)}\Bigr)^{1/2}\Big\vert_{r=R(1+\rho(\varphi))}\,d\varphi,
\end{equation}
where $\nabla^h\rho$ and $\nabla^v\rho$ are the projections of the vector $\nabla\rho\in T_{(R(1+\rho(\varphi)),\varphi)}M$ on $\Dist$ and $\Vist$ respectively. To simplify the exposition, define
\begin{equation}\label{eq:def_phi}
\phi(r):=\int_0^r\tau^{n-1}\omega_g(\tau)\,d\tau,
\end{equation}
and
\begin{equation}\label{eq:def_psi}
\psi(r):=\int_0^r\tau^n\omega_g(\tau)\,d\tau,
\end{equation}
where we recall that $\omega_g(r)$ is the volume density defined in \eqref{eq:def_omega}.
Then, we obtain the formula
\begin{equation}\label{eq:starshaped_vol}
\Vol(E)=\int_{S^{n-1}}\phi(R(1+\rho))\,d\varphi,
\end{equation}
and when the barycenter is at zero
\begin{equation}\label{eq:starshaped_bar}
0=\int_{S^{n-1}}\varphi\psi(R(1+\rho(\varphi)))d\varphi.
\end{equation}
Setting $\rho\equiv 0$, we recover the volume and perimeter of the geodesic ball $\B(R)$:
\[
\Vol(\B(R))=n\om_n\phi(R),\quad\Per(\B(R))=n\om_n\phi'(R).
\]
For example, when $\K=\C$, we can compute 
\[
\Vol(\B(R))=\om_n\sinh^n(R),\quad\Per(\B(R))=n\om_n\sinh^{n-1}(R)\cosh(R).
\]
\subsection{Finite perimeter sets}We recall the definition and some properties of finite perimeter sets in a general Riemannian manifold. We refer to \cite{Maggi_perimeter} for a detailed presentation in the Euclidean space.
\begin{definition}[Sets with finite perimeter]Let $(M,g)$ be a smooth Riemannian manifold with volume element $d\vol$, and $E\subset M$ be a measurable subset. For any open subset $O\subset M$ we will denote with $\Gamma_c(TO)$ the set of smooth vector fields on $M$ compactly supported in $O$. We define the relative perimeter of $E$ in $O$ as
\[
\Per(E,O):=\sup\Bigl\{\int_O\diva^g(\xi)\,d\vol_g,:\xi\in \Gamma_c(TO),\,\sup_{x\in O}g(\xi,\xi)\leq 1\Bigr\}.
\]
If $\Per(E,O)<+\infty$ for all $O\subset\subset M$ we say that $E$ is a set with locally finite perimeter, and if $\Per(E):=\Per(E,M)<+\infty$ we say that $E$ is a set with finite perimeter.
\end{definition}
Letting $D{\chi_E}$ be the distributional gradient of the characteristic function $\chi_E$, then
\[
\Per(E,O)=\abs{D{\chi_E}}(O),
\]
where $\abs{D{\chi_E}}$ denotes the total variation of the measure $D{\chi_E}$. 
\begin{definition}Let $E\subset M$ be a set of locally finite perimeter. We define its reduced boundary as
\[
\partial^*E:=\Bigl\{x\in \spt(\abs{D{\chi_E}}):\exists\,\nu_g(x):=\lim_{r\to 0^+}-\frac{D{\chi_E}(B_x(r))}{\abs{D{\chi_E}(B_x(r))}}\text{ unit tangent vector at }x\Bigr\}.
\]
\end{definition}
The next theorem allows us to express $\Per(E)$ as an integration over the reduced boundary, where $\nu_g(x)$, the measure theoretic outwards unit normal to $E$, is well defined. For the proof we refer to \cite{ambro}.
\begin{theorem}[De Giorgi structure theorem]Let $E\subset M$ be a set with locally finite perimeter. Then,
\[
D{\chi_E}=\nu_g(x)d\,\Haus^{n-1}\lvert_{\partial^*E},\text{ and }P(E)=\abs{D{\chi_E}}(M)=\Haus^{n-1}(\partial^*E).
\]
\end{theorem}
This characterization allows us to generalize Equation \eqref{eq:perimeter} of Lemma \ref{lem:formulae_vol_per} for the class of finite perimeter sets.
\begin{lemma}Let $E\subset M$ be a finite perimeter set. Then, for all open subset $O$ of $M$ we have that
\begin{equation}\label{eq:perimeter_finite}
\Per(E,O)=\int_{\partial ^*E\cap O}\omega_g(r)\Bigl(\abs{\nu^n}^2+\frac{r^2}{\cosh^2(r)\sinh^2(r)}\abs{\nu^h}^2+\frac{r^2}{\sinh^2(r)}\abs{\nu^v}^2\Bigr)^{1/2}\,d\Haus^{n-1},
\end{equation}
where $\nu$ is the measure theoretic outwards unit normal to $E$ with respect to the flat metric $g_e$.
\end{lemma}
\section{Uniform $C^1$-strong stability}
Fix any upper bound $R_0>0$ and a radius $R\in(0,R_0]$ for the perturbed sphere given in normal coordinates
\[
\partial E=\{(R(1+\rho(\varphi)),\varphi):\varphi\in S^{n-1}\}\subset M,
\]
where $\rho\in C^1(S^{n-1},(-1,+\infty))$ is a volume and barycentric preserving perturbation, in the sense that
\[
\Vol(E)=\Vol(\B(R)),
\]
and
\[
\bary(E)=\bary(\B(R))=0.
\]
We suppose 
\[
\norm{\rho}_{C^1}:=\norm{\rho}_{C^1(S^{n-1})}=\sup_{\varphi\in S^{n-1}}\Bigl(\abs{\rho(\varphi)}+\abs{\nabla\rho(\varphi)}\Bigr)\leq\varepsilon,
\]
for some small $\varepsilon\in(0,\frac12)$ yet to define. To simplify the exposition, denote
\[
\bar\rho(\varphi):=R(1+\rho(\varphi)),
\]
and with $\brho:\Sp(R)\to(-1,+\infty)$ the perturbation viewed as a function from the geodesic sphere in $M$, given in normal coordinates as
\[
\brho(R\varphi)=\rho(\varphi), \quad \varphi\in S^{n-1}.
\]
Notice that
\[
\norm{\brho}_{L^2(\Sp(R))}^2:=\norm{\brho}_{L^2(\Sp(R),g)}^2=\int_{S^{n-1}}\abs{\rho(\varphi)}^2\phi'(R)\,d\varphi,
\]
and
\begin{equation}\label{eq:l2_formula}
\norm{\nabla^g\brho}_{L^2(\Sp(R))}^2:=\norm{\nabla^{g}\brho}_{L^2(\Sp(R),g)}^2=\int_{S^{n-1}}\frac{\abs{\nabla^h\rho}^2+\cosh^2(R)\abs{\nabla^v\rho}^2}{\sinh^2(R)\cosh^2(R)}\phi'(R)\,d\varphi.
\end{equation} 
For $k\in\{1,2,3\}$, define the auxiliary functions $\omega_k:\R\to\R$ by
\begin{equation}\label{eq:def_omega}
 \omega_k(r):=\frac{\phi^{(k)}(r)}{r^{n-k}},
\end{equation}
where $\phi(r)$ is defined as in \eqref{eq:def_phi}. Notice that $\omega_1=\omega_g$. We need the next three lemmas to start with the estimates.
\begin{lemma}\label{lem:density_est}For $k\in\{1,2,3\}$, the function $\omega_k$ defined in \eqref{eq:def_omega} is positive, even and strictly convex with removable singularity at zero equal to
\[
\lim_{r\to 0}\omega_k(r)=\begin{cases}1,&k=1,\\ (n-1),&k=2,\\ (n-1)(n-2),&k=3.\end{cases}
\]
In particular, $\min\omega_k=\omega_k(0)$, which is strictly positive, unless $(n,k)=(2,3)$. The constants
\begin{align*}
A_k&:=(n-k)+\max_{r\in[0,R_0]}r\frac{\omega_k'(r)}{\omega_k(r)},\\
B_k&:=\max_{r\in[0,R_0]}\frac{\omega_k(2r)}{\omega_k(r)},\\
C_k&:=\omega_k(2R_0),\\
\end{align*}
are finite, depend only on $(n,d,k,R_0)$, and realize the following inequalities
\begin{align}
\phi^{(k)}_k(R(1+\tau))&\geq (1-A_k\abs{\tau})\phi^{(k)}(R),\label{eq:dis_A}\\
\phi^{(k)}(R(1+\tau))&\leq B_k\phi^{(k)}(R),\label{eq:dis_B}\\
\omega_k(R(1+\tau))&\leq C_k.\label{eq:dis_C}
\end{align}
uniformly in $R\in(0,R_0]$ and $\tau\in \R$, $\abs{\tau}\leq 1$, where $\phi(r)$ is as in \eqref{eq:def_phi}.
\end{lemma}
\begin{proof}
We explicitly compute
\begin{align*}
r^{n-1}\omega_1(r)&=\phi'(r)=\sinh^{n-1}(r)\cosh^{d-1}(r),\\
r^{n-2}\omega_2(r)&=\phi''(r)=(n-1)\sinh^{n-2}(r)\cosh^{d}(r)+(d-1)\sinh^n(r)\cosh^{d-2}(r),\\
r^{n-3}\omega_3(r)&=\phi'''(r)=(n-1)(n-2)\sinh^{n-3}(r)\cosh^{d+1}(r)\\
&\quad+(2dn-d-n)\sinh^{n-1}(r)\cosh^{d-1}(r)+(d-1)(d-2)\sinh^{n+1}(r)\cosh^{d-3}(r).
 \end{align*}
 Since the functions $r^{-1}\sinh(r)$ and $r\sinh(r)$ are even, strictly convex functions with positive (removable singularity) at zero, $\sinh(r)$ is odd and $\cosh(r)$ is even, we can infer that $\omega_k$ is itself convex, even and positive. Developing by Taylor we get that
 \begin{align*}
 r^{n-1}\omega_1(r)&=r^{n-1}+o(r^{n}),\\
 r^{n-2}\omega_2(r)&=(n-1)r^{n-2}+o(r^{n-1}),\\
 r^{n-3}\omega_3(r)&=(n-1)(n-2)r^{n-3}+(2dn-d-n)r^{n-1}+o(r^{n}),
 \end{align*}
proving that $\omega_k$ can be extended at zero with value 1, $(n-1)$ and $(n-1)(n-2)$, according to $k$ being equal to 1, 2 or 3 and that
\[
\lim_{r\to 0}r\frac{\omega_k'(r)}{\omega_k(r)}<+\infty,\quad\text{and }
\lim_{r\to 0}\frac{\omega_k(2r)}{\omega_k(r)}<+\infty.
\]
This shows that $A_k$ and $B_k$ are well defined functions. To prove Equation \eqref{eq:dis_A}, since $\omega_k$ is convex and $(1+\tau)^l\geq 1+l\tau$ for all $\abs{\tau}\leq 1$ and $l\in\Z$, we can estimate
\begin{align*}
\phi^{(k)}(R(1+\tau))&=R^{n-k}(1+\tau)^{n-k}\omega_k(R(1+\tau))\\
&\geq(1+(n-k)\tau)R^{n-k}\bigl(\omega_k(R)+\omega_k'(R)R\tau\bigr)\\
&=(1+(n-k)\tau)\bigl(1+R\frac{\omega_k'(R)}{\omega_k(R)}\tau\bigr)\phi^{(k)}(R)\\
&\geq(1-A_k\abs{\tau})\phi^{(k)}(R).
\end{align*}
Equations \eqref{eq:dis_B} and \eqref{eq:dis_C} are immediate, given the nature of $\omega_k$ and the bound on $\tau$.
\end{proof}
\begin{lemma}\label{lem:gradient_est}For every function $\rho\in C^1(S^1,(-1,+\infty))$ and $0<R\leq R_0$ there exists a constant $D>0$, depending only on $R_0$, such that
\begin{equation}\label{eq:gradient_est}
    \abs{\nabla\rho}^2\geq\bar\rho^2\frac{\abs{\nabla^h\rho}^2+\cosh^2(\bar\rho)\abs{\nabla^v\rho}^2}{\sinh^2(\bar\rho)\cosh^2(\bar\rho)}\geq(1-D\abs{\rho})R^2\frac{\abs{\nabla^h\rho}^2+\cosh^2(R)\abs{\nabla^v\rho}^2}{\sinh^2(R)\cosh^2(R)},
\end{equation}
where $\bar\rho\in C^1(S^{n-1},(0,+\infty))$ was defined as $\bar \rho(\varphi)=R(1+\rho(\varphi))$.
\end{lemma}
\begin{proof}Setting
\begin{align*}
\omega_1^1(r)&:=\frac{\sinh^2(r)}{r^2},\\
\omega_1^2(r)&:=\frac{\sinh^2(r)\cosh^2(r)}{r^2},
\end{align*}
and arguing as in Lemma \ref{lem:density_est}, we notice that $\omega_1^1$ and $\omega_1^2$ are strictly convex, even and equal to one at zero. Setting
\[
\xi(r):=\frac{\abs{\nabla^h\rho}^2}{\omega_1^2(r)}+\frac{\abs{\nabla^v\rho}^2}{\omega_1^1(r)},
\]
we get that
\[
\xi(r)\leq \frac{1}{\omega_1^1(r)}(\abs{\nabla^h\rho}^2+\abs{\nabla^v\rho}^2)\leq\abs{\nabla\rho}^2,
\]
which is the first inequality of Equation \eqref{eq:gradient_est} when $r=\bar\rho$. For the second inequality, for $i=1,2$ we have by convexity of $\omega^i_1$ that
\[
\omega_1^i(R)\geq\omega_1^i(\bar\rho)-(\omega_1^i)'(\bar\rho)R\rho,
\]
implying that
\begin{align*}
\omega_1^i(\bar\rho)&\leq\omega_1^i(R)+(\omega_1^i)'(\bar\rho)R\rho\leq\omega_1^i(R)\Bigl(1+R\frac{(\omega_1^i)'(\bar\rho)}{\omega_1^i(R)}\abs{\rho}\Bigr)\\
&\leq\omega_1^i(R)(1+R_0(\omega_1^i)'(2R_0)\abs{\rho}),
\end{align*}
and
\[
\frac{1}{\omega_1^i(\bar\rho)}\geq\frac{1}{1+R_0(\omega_1^i)'(2R_0)\abs{\rho}}\frac{1}{\omega_1^i(R)}\geq(1-R_0(\omega_1^i)'(2R_0)\abs{\rho})\frac{1}{\omega_1^i(R)}.
\]
Hence, setting $D:=\max_{i\in\{1,2\}}\{R_0(\omega_1^i)'(2R_0)\}$, we can estimate
\[
\xi(\bar\rho)\geq(1-D\abs{\rho})\xi(R),
\]
completing the proof of the Lemma.
\end{proof}
Recall that we denote with $\lambda_1^R$ the first non-zero eigenvalue of the Laplacian operator on the sphere $\Sp(R)$.
\begin{lemma}\label{lem:magic_lambda} Let $\phi(r)$ as in \eqref{eq:def_phi}. For all $r>0$ we have the following identity
\begin{equation}
\phi'''(r)=\phi'(r)\Bigl(\frac{\phi''(r)^2}{\phi'(r)^2}-\lambda_1^r\Bigr),
\end{equation}
where $\lambda_1^r$ is the first eigenvector of the Laplacian on the sphere $\Sp(r)$.
\end{lemma}
\begin{proof}
Since $\phi'(r)=r^{n-1}\omega_g(r)=\sinh^{n-1}(r)\cosh^{d-1}(r)$ we can compute
\begin{align*}
\phi''(r)&=(n-1)\sinh^{n-2}(r)\cosh^d(r)+(d-1)\sinh^n(r)\cosh^{d-2}(r)\\
&=(n-1)\coth(r)\phi'(r)+(d-1)\tanh(r)\phi'(r),
\end{align*}
and
\begin{align*}
\phi'''(r)&=\phi''(r)\Bigl((n-1)\coth(r)+(d-1)\tanh(r)\Bigr)+\phi'(r)\Bigl(\frac{d-1}{\cosh^2(r)}-\frac{n-1}{\sinh^2(r)}\Bigr)\\
&=\frac{\phi''(r)^2}{\phi'(r)}+\phi'(r)\frac{(d-1)\sinh^2(r)-(n-1)\cosh^2(r)}{\cosh^2(r)\sinh^2(r)}\\
&=\frac{\phi''(r)^2}{\phi'(r)}+\phi'(r)\frac{(d-1)-(n-1-d+1)\cosh^2(r)}{\cosh^2(r)\sinh^2(r)}=\frac{\phi''(r)^2}{\phi'(r)}-\phi'(r)\lambda_1^r,
\end{align*}
as wished.
\end{proof}
We are now ready to prove a first estimate.
\begin{proposition}[Intermediate estimate]\label{prop:intermediate}Under the assumptions of Theorem \ref{thm:1} one has that
\begin{equation}\label{eq:intermediate}
\begin{split}
\Per(E)-\Per(\B(R))&\geq-\Bigl(\Bigl(\frac{B_3C_2C_3}{3}+A_3C_2^2\Bigr)\norm{\rho}_{C^0}+R^2\lambda_1^R\Bigr)\frac12\norm{\brho}^2_{L^2(\Sp(R))}\\
&\quad+R^2(1-(D+2)\norm{\rho}_{C^1})\frac{1}{2}\norm{\nabla^g\brho}^2_{L^2(\Sp(R))}.
\end{split}
\end{equation}
Here the constants $A_k$, $B_k$, $C_k$ and $D$ have been defined in Lemma \ref{lem:density_est} and Lemma \ref{lem:gradient_est}.
\end{proposition}
\begin{proof}Setting
\[
\xi(r):=r^2\frac{\abs{\nabla^h\rho}^2+\cosh^2(r)\abs{\nabla^v\rho}^2}{\sinh^2(r)\cosh^2(r)},
\]
we have by Equation \eqref{eq:starshaped_per} that
\[
\Per(E)=\int_{S^{n-1}}\Bigl(1+\frac{\xi(\bar\rho)}
{(1+\rho)^2}\Bigr)^{1/2}\phi'(\bar\rho)\,d\varphi.
\]
By the elementary inequalities
\[
\frac{1}{(1+t)^2}\geq(1-2t),\text{ for all }t>-1,\text{ and }(1+t)^{1/2}\geq 1+\frac{t}{2}-\frac{t^2}{8}\text{ for all }t\geq 0,
\]
we have that
\begin{equation}\label{eq:polar_1}
\begin{split}
\Per(E)&\geq\int_{S^{n-1}}(1+\xi(\bar\rho)(1-2\rho))^{1/2}\phi'(\bar\rho)\,d\varphi\\
&\geq\int_{S^{n-1}}\Bigl(1+(1-2\rho)\Bigl(1-(1-2\rho)\frac{\xi(\bar\rho)}{4}\Bigl)\frac{\xi(\bar\rho)}{2}\Bigr)\phi'(\bar\rho)\,d\varphi\\
&\geq\int_{S^{n-1}}\phi'(\bar\rho)\,d\varphi+\int_{S^{n-1}}\Bigl(1-2\rho-\xi(\bar\rho)\Bigr)\frac{\xi(\bar\rho)}{2}\phi'(\bar\rho)\,d\varphi\\
&\geq\int_{S^{n-1}}\phi'(\bar\rho)\,d\varphi+(1-2\norm{\rho}_{C^1})(1-D\norm{\rho}_{C^0})\int_{S^{n-1}}\frac{\xi(R)}{2}\phi'(\bar\rho)\,d\varphi\\
&\geq \int_{S^{n-1}}\phi'(\bar\rho)\,d\varphi+(1-(D+2)\norm{\rho}_{C^1})\int_{S^{n-1}}\frac{\xi(R)}{2}\phi'(\bar\rho)\,d\varphi\\
&=\int_{S^{n-1}}\phi'(\bar\rho)\,d\varphi+(1-(D+2)\norm{\rho}_{C^1})\frac{1}{2}R^2\norm{\nabla^g\brho}_{L^2(\Sp(R))}^2,
\end{split}
\end{equation}
where at the end we took advantage of Equations \eqref{eq:l2_formula} and \eqref{eq:gradient_est}. We need to treat the first term
\[
\int_{S^{n-1}}\phi'(\bar\rho)\,d\varphi.
\]
Now, by Taylor expansion there exists $\eta:S^{n-1}\to[0,1]$ such that
\begin{equation*}\label{eq:taylor1}
\phi'(\bar\rho)=\phi'(R)+\phi''(R)R\rho+\frac{R^2\rho^2}{2}\phi'''(R(1+\eta\rho)). 
\end{equation*}
On the other side, thanks to Lemma \ref{lem:density_est} and Lemma \ref{lem:magic_lambda} we have that
\begin{align*}
\phi'''(R(1+\eta\rho))&\geq (1-A_3\norm{\rho}_{C^0})\phi'''(R)=(1-A_3\norm{\rho}_{C^0})\Bigl(\frac{\phi''(R)^2}{\phi'(R)^2}-\lambda_1^R\Bigr)\phi'(R)\\
&\geq\frac{\phi''(R)^2}{\phi'(R)}-\Bigl(A_3\Bigl(\frac{R^{n-2}\omega_2(R)}{R^{n-1}\omega_1(R)}\Bigr)^2\norm{\rho}_{C^0}+\lambda_1^R\Bigl)\phi'(R)\\
&\geq\frac{\phi''(R)^2}{\phi'(R)}-\Bigl(A_3C_2^2R^{-2}\norm{\rho}_{C^0}+\lambda_1^R\Bigl)\phi'(R),
\end{align*}
which gives the following estimate
\begin{equation}\label{eq:polar_2}
\begin{split}
\int_{S^{n-1}}\phi'(\bar\rho)\,d\varphi&\geq\Per(\B(R))+\int_{S^{n-1}}\phi''(R)R\rho+\frac{\phi''(R)^2}{\phi'(R)}\frac{R^2\rho^2}2\,d\varphi\\
&\quad-\Bigl(A_3C_2^2\norm{\rho}_{C^0}+R^2\lambda_1^R\Bigl)\int_{S^{n-2}}\frac{\rho^2}{2}\phi'(R)\,d\varphi.
\end{split}
\end{equation}
Now, by the volume preserving constraint over $\rho$, we can integrate the Taylor expansion
\begin{equation}\label{eq:Taylor_for_volume}
\phi(\bar\rho)-\phi(R)=\phi'(R)R\rho+\frac{R^2\rho^2}{2}\phi''(R)+\frac{R^3\rho^3}{6}\phi'''(R(1+\bar\eta\rho)),
\end{equation}
where $\bar\eta:S^{n-1}\to[0,1]$ is suitably chosen, to obtain
\begin{equation*}\begin{split}
&\abs*{\int_{S^{n-1}}\phi'(R)R\rho +\frac{R^2\rho^2}{2}\phi''(R)\,d\varphi}=\abs*{\int_{S^{n-1}}\frac{R^3\rho^3}{6}\phi'''(R(1+\bar\eta\rho))\,d\varphi}\\
&\quad\leq \norm{\rho}_{C^0}B_3\frac{R}{3}\int_{S^{n-1}}\frac{\rho^2}{2}R^2\phi'''(R)\,d\varphi\\
&\quad= \norm{\rho}_{C^0}B_3\frac{R}{3}\int_{S^{n-1}}\frac{\rho^2}{2}R^2\frac{\phi'''(R)}{\phi'(R)}\phi'(R)\,d\varphi\\
&\quad=\norm{\rho}_{C^0}B_3\frac{R}{3}\int_{S^{n-1}}\frac{\rho^2}{2}R^2\frac{R^{n-3}\omega_3(R)}{R^{n-1}\omega_1(R)}\phi'(R)\,d\varphi\\
&\quad\leq\norm{\rho}_{C^0}B_3C_3\frac{R}{3}\int_{S^{n-1}}\frac{\rho^2}{2}\phi'(R)\,d\varphi.
\end{split}
\end{equation*}
This precious estimate allows us to treat \eqref{eq:polar_2}:
\begin{align*}
&\int_{S^{n-1}}\phi''(R)R\rho+\frac{\phi''(R)^2}{\phi'(R)}\frac{R^2\rho^2}{2}\,d\varphi\geq-\abs*{\frac{\phi''(R)}{\phi'(R)}}\abs*{\int_{S^{n-1}}\phi'(R)R\rho +\frac{R^2\rho^2}{2}\phi''(R)\,d\varphi}\\ 
&\quad\geq -\norm{\rho}_{C^0}B_3C_3\frac{R}{3}\abs*{\frac{\phi''(R)}{\phi'(R)}}\int_{S^{n-1}}\frac{\rho^2}{2}\phi'(R)\,d\varphi,\\
&\quad\geq -\norm{\rho}_{C^0}B_3C_3\frac{R}{3}R^{-1}C_2\int_{S^{n-1}}\frac{\rho^2}{2}\phi'(R)\,d\varphi,\\
&\quad=-\norm{\rho}_{C^0}\frac{B_3C_2C_3}{3}\norm{\brho}_{L^2(\Sp(R))}^2,
\end{align*}
which combined with \eqref{eq:polar_1} and \eqref{eq:polar_2} gives the desired inequality.
\end{proof}
We are now ready to prove Theorem \ref{thm:1}.
\begin{proof}[Proof of Theorem \ref{thm:1}]Recall from Section \ref{sec:distribution_spectral_decomposition}, that there exists an orthogonal decomposition of $L^2(\Sp(R),g)$ in spherical harmonics of the form
\[
\{f^R_{j,k}\in L^2(\Sp(R),g):1\leq k\leq n_j\}_{j\geq 0}.
\]
Choosing the renormalization so that
\[
\frac{1}{\Per(\B(R))}\norm{f^R_{j,k}}_{L^2(\Sp(R))}^2=1,
\]
$f^R_0\equiv 1$, and the eigenspace associated to $\lambda^R_1$ is spanned by restricting on $\Sp(R)$ the harmonic polynomials of degree one
\[
f_{1,k}^R(x):=\sqrt{n}\frac{x^k}{R},\quad k=1,\dots,n,
\]
given in the cartesian coordinates chart $x=R\varphi$. We develop $\brho$ on the spherical harmonics of $\Sp(R)$ by setting for all $j\geq 0$ and $1\leq l\leq n_j$ the coefficients
\[
c_{j,k}:=\frac{1}{\Per(\B(R))}\inn{\brho,f^R_{j,k}}_{L^2(\Sp(R))},
\]
so that
\[
\frac{1}{\Per(\B(R))}\norm{\brho}^2_{L^2(\Sp(R))}=\sum_{j\geq 0}\sum_{k=1}^{n_j}c_{j,k}^2,
\]
and
\[
\frac{1}{\Per(\B(R))}\norm{\nabla^g\brho}^2_{L^2(\Sp(R))}=\sum_{j\geq 1}\sum_{k=1}^{n_j}\lambda^R_jc_{j,k}^2.
\]
To simplify the exposition, we will write $\sum_{j,k}$ instead of the double sums, and $\fint_{S^{n-1}}$ the mean with respect to $(\Sp(R),g)$, that is $(\Per(\B(R)))^{-1}\int_{S^{n-1}}$.  To take advantage of the spectral gap to control the negative term in \eqref{eq:intermediate}, we have to estimate the zero and first harmonics $c_0$ and $c_{1,k}$. From Equation \eqref{eq:Taylor_for_volume}, the volume preservation implies the following estimate
\begin{equation}\label{eq:bound_harmo1}
\begin{split}
\abs*{\int_{S^{n-1}}\rho\phi'(R)d\varphi}&=\abs*{\int_{S^{n-1}}\frac{R\rho^2}{2}\phi''(R)+\frac{R^2\rho^3}{6}\phi'''(R(1+\rho\eta))\,d\varphi}\\
&\leq(C_2+C_3)\int_{S^{n-1}}\frac{\rho^2}{2}\,\phi'(R)\,d\varphi,
\end{split}
\end{equation}
allowing us by Cauchy-Schwarz inequality to treat the first harmonic as
\begin{equation}
\begin{split}
c_0^2=\abs*{\fint_{S^{n-1}}\rho\phi'(R)\,d\varphi}^2\leq(C_2+C_3)^2\norm{\rho}^2_{C^0}\fint_{S^{n-1}}\frac{\rho^2}{2}\phi'(R)\,d\varphi.
\end{split}
\end{equation}
The barycentric preservation and Equation \eqref{eq:starshaped_bar} give us the analogue for the second harmonics: first, by Taylor approximation there exists $\theta:S^{n-1}\to[0,1]$ such that
\begin{align*}
\psi(\bar\rho)-\psi(R)&=\psi'(R)R\rho+\psi''(R(1+\theta\rho))\frac{R^2\rho^2}{2}\\
&=\phi'(R)R^2\rho+\bigl(\phi'(R(1+\theta\rho))+R(1+\theta\rho)\phi''(R(1+\theta\rho))\bigr)\frac{R^2\rho^2}{2}.
\end{align*}
Then, for any $k\in\{1,\dots,n\}$, we have that
\begin{align*}
\abs*{\int_{S^{n-1}}\frac{x^k}{R}\rho\phi'(R)\,dx}&=\abs*{\int_{S^{n-1}}\frac{x^k}{R}\frac{\rho^2}{2}\Bigl(\phi'(R(1+\theta\rho))+R(1+\theta\rho)\phi''(R(1+\theta\rho))\Bigr)\,dx}\\
&\leq(B_1+2C_2)\int_{S^{n-1}}\frac{\rho^2}{2}\phi'(R)\,dx,
\end{align*}
which implies by Cauchy-Schwarz inequality that
\begin{equation}\label{eq:bound_harmo2}
\begin{split}
\sum_{k=1}^{n}c_{1,k}^2&=\sum_{k=1}^n\Bigl(\fint_{S^{n-1}}\sqrt{n}\frac{x^k}{R}\rho\phi'(R)\,d\varphi\Bigr)^2\leq n^2(B_1+2C_2)^2\Bigl(\fint_{S^{n-1}}\frac{\rho^2}{2}\phi'(R)\,d\varphi\Bigr)^2\\
&\leq \norm{\rho}_{C^0}^2n^2(B_1+2C_2)^2\Bigl(\fint_{S^{n-1}}\frac{\abs{\rho}}{2}\phi'(R)\,d\varphi\Bigr)^2\\
&\leq\norm{\rho}_{C^0}^2n^2(B_1+2C_2)^2\frac12\fint_{S^{n-1}}\frac{\rho^2}{2}\phi'(R)\,d\varphi.
\end{split}
\end{equation}
Combining \eqref{eq:bound_harmo1} and \eqref{eq:bound_harmo2} we obtain that
\begin{equation}
c_0^2+\sum_{k=1}^n c_{1,k}^2\leq \Bigl((C_2+C_3)^2+\frac{n^2}{2}(B_1+2C_2)^2\Bigr)\norm{\rho}_{C^0}\sum_{j,k}c_{j,k}^2.
\end{equation}
Set $K_1:=(C_2+C_3)^2+\frac{n^2}{2}(B_1+2C_2)^2$. We have in particular that
\begin{align*}
\frac{1}{\Per(\B(R))}\norm{\brho}^2_{L^2(\Sp(R))}=\sum_{j,k}c_{j,k}^2&\leq \frac{1}{1-K_1\norm{\rho}_{C^0}}\sum_{j\geq 2,k}c_{j,k}^2\\
&\leq\frac{1}{\lambda_2^R(1-K_1\norm{\rho}_{C^0})}\sum_{j\geq 2,k}\lambda_j^Rc_{j,k}^2\\
&\leq\frac{1}{\lambda^R_2(1-K_1\norm{\rho}_{C^0})}\frac{1}{\Per(\B(R))}\norm{\nabla^g\brho}^2_{L^2(\Sp(R))}.
\end{align*}
Plugging this last key estimate in Eqaution \eqref{eq:intermediate} of Proposition \ref{prop:intermediate}
\begin{equation}
\begin{split}
\Per(E)-\Per(\B(R))&\geq-\Bigl(K_2\norm{\rho}_{C^0}+R^2\lambda_1^R\Bigr)\frac12\norm{\brho}^2_{L^2(\Sp(R))}\\
&\quad+R^2(1-K_3\norm{\rho}_{C^1})\frac{1}{2}\norm{\nabla^g\brho}^2_{L^2(\Sp(R))},
\end{split}
\end{equation}
where $K_2=\frac{B_3C_2C_3}{3}+A_3C_2^2$ and $K_3=D+2$,
we are finally able to control the negative term taking advantage of the spectral gap between the two first harmonics. In fact, one can check that
\[
\frac{\lambda_1^R}{\lambda_2^R}\leq\frac{(d-1)+(n-d)\cosh^2(R)}{2d+2(n-d)\cosh^2(R)}<\frac12,
\]
uniformly in $R$, and therefore supposing
\[
\norm{\rho}_{C^0}\leq\frac{1}{3K_1},
\]
so that
\[
\frac{\lambda_1^R}{\lambda_2^R(1-K_1\norm{\rho}_{C^1})}<\frac{1}{2}\frac{1}{(1-K_1\norm{\rho}_{C^1})}\leq\frac34,
\]
we can estimate
\begin{align*}
\Per(E)&-\Per(\B(R))\\
&\geq-K_2\norm{\rho}_{C^0}\frac12\norm{\brho}^2_{L^2}+R^2\Bigl(1-\frac{\lambda^R_1}{\lambda^R_2(1-K_1\norm{\rho}_{C^0})}-K_3\norm{\rho}_{C^1}\Bigr)\frac{1}{2}\norm{\nabla^g\brho}^2_{L^2}\\
&\geq-K_2\norm{\rho}_{C^0}\frac12\norm{\brho}^2_{L^2}+R^2\Bigl(\frac14-K_3\norm{\rho}_{C^1}\Bigr)\frac{1}{2}\norm{\nabla^g\brho}^2_{L^2}\\
&\geq\Bigl(\frac{R^2\lambda^R_2}{12}-K_2\norm{\rho}_{C^0}\Bigr)\frac12\norm{\brho}^2_{L^2}+R^2\Bigl(\frac18-K_3\norm{\rho}_{C^1}\Bigr)\frac{1}{2}\norm{\nabla^g\brho}^2_{L^2}.
\end{align*}
Finally, if
\[
\norm{\rho}_{C^1}<\varepsilon:=\min\Bigl\{\frac12,\frac{1}{3K_1},\frac{R^2\lambda^R_2}{24K_2},\frac{1}{16K_3}\Bigr\},
\]
we obtain the desired inequality
\begin{align*}
\Per(E)-\Per(\B(R))&\geq\frac{R^2\lambda_2^R}{48}\norm{\brho}_{L^2(\Sp(R))}^2+\frac{R^2}{32}\norm{\nabla^g\brho}_{L^2(\Sp(R))}^2.
\end{align*}
We are left to prove that $\varepsilon>0$ can be chosen uniformly in $R$, that is
\[
R\mapsto R^2\lambda^R_2,
\]
is uniformly bounded away from zero in $[0,R_0]$, being all other constants already depending only on $(n,d,R_0)$. This is an consequence of the exact explicit form of the eigenvalues of the Laplacian on $\Sp(R)$, that can be expressed as
\[
\frac{a\cosh^2(R)+b}{\sinh^2(R)\cosh^2(R)},
\]
for some coefficients $a,b\in\N$, see \cite[Theorem A]{bettiol2022laplace}. In particular, $R\mapsto R^2\lambda_2^R$ is uniformly bounded away from zero in $[0,R_0]$, as wished.
\end{proof}
\section{Minimality of balls in small volume regime}
This section is devoted to the proof of Theorem \ref{thm:small_volumes}. We proceed in three steps: first we show that small isoperimetric sets are uniformly almost-minimizing. We recall that a set $E\subset M$ is almost-minimizing if it is optimal up to an error uniformly proportional to the size of the perturbation. In our case, given an isoperimetric region $E$ with volume $v$, this translates to the existence of a universal constant $C>0$ such that
\begin{equation*}
\Per(E,\B(x,s))\leq \Per(F,\B(x,s))+\frac{C}{v^{1/n}}\Vol(E\triangle F),
\end{equation*}
whenever $E\triangle F\subset \B(x,s)$ and $s\leq s_1=s_1(v)$. Then, we prove that this condition combined with the strong stability results in the Euclidean space, imply the $L^1$ and $L^\infty$-proximity to a geodesic ball with respect to the induced Euclidean metric $g_e$ when $o$ equal to the barycenter of $E$. Since almost-minimizing sets sufficiently close to a smooth surface are $C^{1,\alpha}$-normal perturbations of it, we conclude the argument by applying Theorem \ref{thm:1}.
\subsection{Almost-minimality}
For any subset $G\subset M$, denote the dilation by $\tau>0$ with respect to the normal coordinates $(r,\varphi)$ pointed at $o\in M$ with
\[
\tau G:=\{(\tau r,\varphi)\in M:(r,\varphi)\in G\}.
\]
Recall that we denote with $\PER(\cdot)$ and $\VOL(\cdot)$ the perimeter and volume functionals with respect to the Euclidean metric $g_e$. We prove the following estimates.
\begin{lemma}\label{lem:rescaling}Let $G\subset M$ be a finite perimeter set contained in $\B(o,R)$ for some $R>0$. Then, there exists $C=C(n,d,R)>0$ such that for all $t\in[0,1]$ the following estimates on the volume and perimeter of the dilation by $(1+t)$ hold
\begin{equation}\label{eq:resc_vol}
(1+t)^n\Vol(G)\leq\Vol((1+t)G)\leq (1+Ct)\Vol(G),
\end{equation}
and
\begin{equation}\label{eq:resc_per}
\Per((1+t)G)\leq (1+Ct)\Per(G).
\end{equation}
Moreover, one has that
\begin{equation}\label{eq:comparison_volume}
    \VOL(G)\leq\Vol(G)\leq(1+\omega_g'(R)R)\VOL(G),
\end{equation}
and
\begin{equation}\label{eq:comparison_perimeter}
    \PER(G)\leq\Per(G)\leq(1+\omega_g'(R)R)\PER(G).
\end{equation}
\end{lemma}
\begin{proof}Thanks to Equation \ref{eq:volume} we can express
\begin{align*}
\Vol((1+t)G)&=\int_{(1+t)G}\omega_g(r)\,d\Haus^n=(1+t)^n\int_G\omega_g((1+t)r)\,d\Haus^n.
\end{align*}
The first inequality of Equation \eqref{eq:resc_vol} is immediate from the fact that $\omega_g$ is monotone. On the other side, arguing as in Lemma \ref{lem:density_est} we have that by the convexity of $\omega_g$ we can estimate
\[
\omega_g((1+t)r)\leq\Bigl(1+r\frac{\omega'_g(r(1+t))}{\omega_g(r)}t\Bigr)\omega_g(r)\leq \Bigl(1+R\omega'_g(2R)t\Bigr)\omega_g(r),
\]
proving that 
\begin{align*}
\Vol((1+t)G)&\leq (1+t)^n(1+R\omega'_g(2R)t)\Vol(G)\\
&\leq (1+(2^n-1)t)(1+R\omega'_g(2R)t)\Vol(G)\leq (1+C_1t)\Vol(G),
\end{align*}
 for $C=(2^n+2)(R\omega_g'(2R)+1)$, as wished. Taking advantage of the integral representation of the perimeter \eqref{eq:perimeter_finite}, we have that $\Per((1+t)G)$ is equal to
\begin{equation*}
(1+t)^{n-1}\int_{\partial ^*G}\omega_g(r(1+t))\Bigl(\abs{\nu^n}^2+\frac{1}{\omega_1^2(r(1+t))}\abs{\nu^h}^2+\frac{1}{\omega_1^1(r(1+t))}\abs{\nu^v}^2\Bigr)^{1/2}\,d\Haus^{n-1},
\end{equation*}
where $\omega^1_1(r)^{-1}=\frac{r^2}{\sinh^2(r)}$ and $\omega^2_1(r)^{-1}=\frac{r^2}{\sinh^2(r)\cosh^2(r)}$, are decreasing functions. We conclude that
\[
\Per((1+t)G)\leq(1+t)^{n-1}\bigl(1+\omega_g'(2R)t\bigr)\Per(G)\leq (1+Ct)\Per(G),
\]
as wished. Equations \eqref{eq:comparison_volume} and \eqref{eq:comparison_perimeter} are obtained analogously.
\end{proof}
Before proving the almost-minimality of isoperimetric sets with small volume, we need to state two important results.
\begin{proposition}\label{prop:diameter_decay}There exist $\bar v=\bar v(n,d)>0$ and $\mu=\mu(n,d)>0$ such that
\[
\diam(E)\leq \mu \Vol(E)^{1/n},
\]
whenever $E$ is an isoperimetric set with volume $\Vol(E)\leq \bar v$. 
\end{proposition}
\begin{proof}This result holds in all generality for manifolds with uniform bound on the Ricci curvature from below and positive injectivity radius. See \cite[Lemma 4.9]{Nardulli} and the recent paper \cite[Proposition 4.23]{Antonelli} for the very general case of RCD spaces.
\end{proof}
\begin{lemma}\label{lem:isoperimetric_type}Fix $v_0>0$. Then, there exists $C_0=C_0(v_0,n)>0$ such that for any finite perimeter set $E$ with $\Vol(E)\leq v_0$ one has that
\[
C_0\Per(E)\geq \Vol(E)^{(n-1)/n}.
\]
\end{lemma}
\begin{proof}This result holds in all generality for manifolds with bounded Ricci curvature from below. See \cite[Lemma 3.5]{Hebey}, and \cite[Lemma 3.10]{GR} for an alternative proof in the general setting of sub-Riemannian manifolds.
\end{proof}
From now on, we will fix $v_0=\bar v>0$, $\mu>0$ and $C_0>0$ as in the statement of Proposition \ref{prop:diameter_decay} and Lemma \ref{lem:isoperimetric_type}.
We are now ready to prove that the isoperimetric sets are almost-minimizers uniformly in $0<v\leq\bar v$.
\begin{proposition}[Almost-minimality in $M$]\label{prop:almost_minimality_M}There exists $C_1=C_1(\bar v,n,d)>0$ such that if $E$ is an isoperimetric set and $\Vol(E)=v\leq \bar v$, then
\begin{equation}\label{eq:almost_min_M}
\Per(E,\B(x,s))\leq \Per(F,\B(x,s))+\frac{C_1}{v^{1/n}}\Vol(E\triangle F),
\end{equation}
whenever
\[
0<s\leq s_1=\min\Bigl\{1,\phi^{-1}\Bigl(\frac{v}{2n\omega_n}\Bigr)\Bigr\},
\]
and $F\subset M$ is such that $F\triangle E\subset \B(x,s)$. In particular,
\begin{equation}\label{eq:almost_min_M_2}
\Per(E,\B(x,s))\leq \Per(F,\B(x,s))+\frac{C_1}{v^{1/n}}\Per(E\triangle F)\phi(s)^{1/n}.
\end{equation}
\end{proposition}
\begin{proof}Since the isoperimetric profile $I_M(v):=\min\{\Per(G):\Vol(G)=v\}$ is increasing (see the article of Hsiang \cite[Lemma 3]{Hsiang}), we can suppose without loss of generality that
\[
0\leq\Vol(E)-\Vol(F)\leq\Vol(\B(s))=n\om_n\phi(s).
\]
In particular, imposing $s\leq\phi^{-1}(v/(2n\om_n))$ we have that $\Vol(F)\leq v/2$. Also, notice that we can suppose $\Per(F,\B(x,s))\leq\Per(\B(x,s))$, because otherwise Equation \eqref{eq:almost_min_M} is satisfied since
\begin{align*}
\Per(E,\B(x,s))&\leq\Per(E\cup\B(x,s))-\Per(E,M\setminus \B(x,s))\\
&\leq\Per(\B(s))\leq\Per(F,\B(x,s)).
\end{align*}
Let $o$ be any point in $E$. Then, by Proposition \ref{prop:diameter_decay}, $E\subset \B(o,\mu v^{1/n})$. Therefore, for $s\leq 1$, we can suppose without loss of generality that $F\subset \B(o,\mu \bar v^{1/n}+2)$, because otherwise $\Per(E,\B(x,s))=0\leq \Per(F,\B(x,s))$. By Lemma \ref{lem:rescaling}, Equation \eqref{eq:resc_vol}, there exists $t^*\in(0,1]$ such that
\[
\Vol((1+t^*)F)=v,
\]
where the dilation is taken with respect to the normal coordinates based at $o\in M$. On the other side, minimality of $E$ and Equation \eqref{eq:resc_per} imply that
\[
\Per(E)\leq\Per((1+t^*)F)\leq(1+Ct^*)\Per(F),
\]
and therefore, for almost every $0<s\leq s_1:=\min\{1,\phi^{-1}(v/(2n\om_n))\}$ we have that
\begin{align*}
\Per&(E,\B(x,s))\\
&\leq(1+Ct^*)\Per(F,\B(x,s))+Ct^*\Per(E,M\setminus \B(x,s))\\
&\leq \Per(F,\B(x,s))+Ct^*\bigl(\Per(F,\B(x,s))+I_M(v)\bigr)\\
&\leq \Per(F,\B(x,s))+Ct^*\bigl(n\om_n\phi'(s)+I_M(v)\bigr).
\end{align*}
We notice that Lemma \ref{lem:isoperimetric_type} implies that $C_0I_M(v)\leq v^{(n-1)/n}$, and by monotonicity of $\phi'$, $\phi'(s)\leq\phi'(\phi^{-1}(2v/(n\om_n)))\leq \bar Cv^{(n-1)/n}$, for some constant depending only on $(n,d,\bar v)$ (to see this, look at Taylor expansions in the proof of Lemma \ref{lem:density_est}). Hence
\[
\Per(E,\B(x,s))\leq\Per(F,\B(x,s))+Cv^{(n-1)/n}t^*.
\]
We are left to find an upper bound for $t^*$. Since
\begin{align*}
\Vol(E)-\Vol(F)&\leq\Vol(E\triangle F),
\end{align*}
and by Equation \eqref{eq:resc_per}
\[
\Vol(E)-\Vol(F)=\Vol((1+t^*)F)-\Vol(F)\geq((1+t^*)^n-1)\Vol(F)\geq\frac{nv}{2}t^*,
\]
we get that
\[
t^*\leq \frac{2\Vol(E\triangle F)}{nv},
\]
proving Equation \eqref{eq:almost_min_M}. Equation \eqref{eq:almost_min_M_2} follows from Lemma \ref{lem:isoperimetric_type} observing that
\begin{align*}
\Vol(E\triangle F)&=\Vol(E\triangle F)^{(n-1)/n}\Vol(E\triangle F)^{1/n}\leq C_0\Per(E\triangle F)\Vol(\B(x,s))\\
&=C_0\Per(E\triangle F)n\om_n\phi(s).
\end{align*}
\end{proof}
\subsection{$L^1$ and $L^\infty$-proximity to a geodesic ball}
We prove first that for small enough volumes, isoperimetric sets are $L^1$-close to geodesic balls with respect to the Euclidean metric $g_e$. Then, the almost-minimality of Proposition \ref{prop:almost_minimality_M} improve this to $L^\infty$ by rescaling. From now on, we will always suppose that
\begin{itemize}
    \item[--]$E\subset M$ is an isoperimetric set with small volume $\Vol(E)=v\leq\bar v$, where $\bar v>0$ is as in Proposition \ref{prop:diameter_decay}.
    \item[--]There exists $o\in M$ so that $E\subset \B(o,\mu v^{1/n})$ in virtue of Proposition \ref{prop:diameter_decay}. We will say that a point $p\in M$ is \emph{admissible} if $E\subset\B(p,\mu v^{1/n})$. 
    \item[--]The Euclidean metric $g_e$ (and its associated geometric concepts $B^n(x,s)$, $\VOL(\cdot)$, $\PER(\cdot)$, etc) is the one arising from the normal coordinates pointed at $o\in M$.
\end{itemize}
\begin{proposition}\label{prop:L1_proximity}Let $E\subset M$ be an isoperimetric set of volume $\Vol(E)=v\leq \bar v$. Consider $g_e$ to be the Euclidean metric associated to the normal coordinates pointed at some point $o\in M$, so that $E\subset \B(o,\mu v^{1/n})$. Then, there exists a constant $C=C(n,d,\bar v)>0$ and a point $\tilde x=\tilde x(o)\in M$ such that
\begin{equation}\label{eq:quantitative_comparison_strong}
Cv^{1/n}\geq\Bigl(\frac{\VOL(E\triangle B^n(\tilde x,\tilde t))}{\VOL(E)}\Bigr)^2,
\end{equation}
where $\tilde t>0$ is so that $\VOL(E)=\VOL(B^n(\tilde t))$. In particular
\begin{equation}\label{eq:quantitative_comparison_strong_2}
    V(E\triangle B^n(\tilde x,\tilde t))\leq Cv^{1+1/2n}.
\end{equation}
\end{proposition}
\begin{proof}
We start by proving that
\[
\phi'(s)\leq\cosh(s)^{d-1/n}(n\phi(s))^{(n-1)/n}.
\]
In fact
\begin{align*}
\phi(s)=\int_0^s\phi'(\tau)\,d\tau&=\int_0^s\sinh^{n-1}(\tau)\cosh^{d-1}(\tau)\,d\tau\geq\frac{1}{\cosh(s)}\int_0^s\sinh^{n-1}(\tau)\cosh(\tau)\,d\tau\\
&=\frac{1}{n\cosh(s)}\sinh^n(s),
\end{align*}
implies that
\[
\frac{\phi'(s)}{(n\phi(s))^{(n-1)/n}}\leq\cosh(s)^{d-1+(n-1)/n}=\cosh(s)^{d-1/n}.
\]
Let $\tilde s>0$ so that
\[
\Vol(E)=\Vol(\B(\tilde s)).
\]
Then, by Lemma \ref{lem:rescaling}, Equations \eqref{eq:comparison_volume} and \eqref{eq:comparison_perimeter} we get that
\begin{align*}
\PER(E)&\leq\Per(E)\leq\Per(\B(s))=n\om_n\phi'(\tilde s)\\
&\leq n\om_n^{1/n}\cosh(\tilde s)^{d-1/n}(n\om_n\phi(\tilde s))^{(n-1)/n}\\
&=n\om_n^{1/n}\cosh(\tilde s)^{d-1/n}\Vol(E)^{(n-1)/n}\\
&\leq n\om_n^{1/n}\cosh(\tilde s)^{d-1/n}(1+\omega'(\mu\bar v^{1/n})\mu\bar v^{1/n})\VOL(E)^{(n-1)/n}\\
&\leq n\om_n^{1/n}(1+Cv^{1/n})\VOL(E)^{(n-1)/n},
\end{align*}
since $\tilde s\leq\mu v^{1/n}$. By the quantitative strong isoperimetric inequality in $\R^n$, see \cite{FigalliMaggi}, Equation \eqref{eq:quantitative_comparison_strong} follows immediately. Equation \eqref{eq:quantitative_comparison_strong_2} is a consequence of the fact that $\VOL(E)\leq\Vol(E)\leq v$.
\end{proof}
We argue now by rescaling. Let
\[
\lambda^n:=\VOL(E)=\VOL(B^n(\tilde t))=\om_n\tilde t^n,
\]
and define the rescaled volume and perimeter operators
\begin{align*}
\Vol^*(G^*)&:=\lambda^{-n}\Vol(\lambda G^*),\\
\Per^*(G^*)&:=\lambda^{-(n-1)}\Per(\lambda G^*).
\end{align*}
Set $E^*:=\lambda^{-1}E$. Then we have immediately that the set $E^*$ is renormalized with respect to $g_e$, that is
\[
\VOL(E^*)=1.
\]
Moreover, the $L^1$-proximity is uniformly given by 
\begin{equation}\label{eq:L1_rescaled}
\Vol^*(E^*\triangle B^n(\lambda^{-1}\tilde x,\om_n^{-1/n}))\leq Cv^{1/2n},
\end{equation}
in virtue of Proposition \ref{prop:L1_proximity} and Lemma \ref{lem:rescaling}, Equation \eqref{eq:comparison_volume}. Finally, there exist $C^*>0$ and $s^*>0$ depending only on $(n,d,\bar v)>0$ such that
\begin{equation}\label{eq:rescaled_almost_min}
\Per^*(E^*,B^n(x,s))\leq \Per^*(F^*,B^n(x,s))+C^*\Vol^*(E^*\triangle F^*),
\end{equation}
whenever $E^*\triangle F^*\subset B^n(x,s)$ and $0<s\leq s^*$. This is a consequence of Proposition \ref{prop:almost_minimality_M} and the bounds on the sectional curvature, giving the existence of $\Lambda=\Lambda(n,d,\bar v)>0$ such that $B^n(x,s)\subset\B(x,\Lambda s)$ provided $\dist(o,x)\leq 2\mu \bar v^{1/n}+2$, $0<s<1$.
\begin{proposition}
\label{prop:Linfty_proximity}Let $E\subset M$ be an isoperimetric set of volume $\Vol(E)=v<\bar v$, and $\lambda>0$ such that $E^*:=\lambda^{-1}E$ has Euclidean volume equal to one. There exists $\tilde x(o)\in M$ and $c_1=c_1(n,\bar v)>0$ such that
\[
B^n(\lambda^{-1}\tilde x,(1-c_1 v^{1/2n^2})\om_n^{-1/n})\subset E^*\subset B^n(\lambda^{-1}\tilde x,(1+c_1 v^{1/2n^2})\om_n^{-1/n}).
\]
\end{proposition}
\begin{proof}Let $x\in \partial E^*$, and $h>0$ be the Euclidean distance of $x$ to $\partial B(\lambda^{-1}\tilde x,\om_n^{-1/n})$. For $0<r<\min\{h/2,\tilde s\}$, define the function
\[
W(r):=\Vol^*(E^*\cap B^n(x,r)).
\]
Since $(E^*\cap B^n(x,r))\subset (E^*\triangle B^n(\lambda^{-1}\tilde x,\om_n^{-1/n}))$, we have thanks to Equation \eqref{eq:L1_rescaled} that
\[
W(r)\leq Cv^{1/2n}.
\]
On the other side, setting $F^*:=E^*\setminus B^n(x,r)$, the uniform almost-minimality gives
\[
\Per^*(E^*)\leq\Per^*(F^*)+C^*W(r).
\]
This, with Lemma \ref{lem:isoperimetric_type}, imply that
\begin{align*}
C_0^{-1}W(r)^{(n-1)/n}&\leq \Per^*(E^*\cap B^n(x,r))\leq\Per^*(E^*,B^n(x,r))+\Per^*(E^*\cap\partial B^n(x,r))\\
&\leq C^*W(r)+2\Per^*(E^*\cap\partial B^n(x,r))\\
&=C^*W(r)+2W'(r).
\end{align*}
This shows that there exists $\bar C>0$ such that
\[
\bar Cr^n\leq W(r)\leq Cv^{1/2n},
\]
implying
\[
0<r<\min\{s^*,Cv^{1/2n^2}/\bar C\},
\]
showing, up to taking $\bar v>0$ small enough, that $h\leq 2Cv^{1/2n^2}/\bar C$. This proves that there exists $c_1>0$ such that $\partial E^*$ is contained in the annulus 
\[
A:=B^n(\lambda^{-1}\tilde x,(1+c_1v^{1/2n^2})\om_n^{-1/n})\setminus B^n(\lambda^{-1}\tilde x,(1-c_1v^{1/2n^2})\om_n^{-1/n}).
\]
The $L^1$-proximity \eqref{eq:L1_rescaled} implies that $E^*$ contains the ball $B^n(\lambda^{-1}\tilde x,(1-c_1v^{1/2n^2})\om_n^{-1/n})$, because otherwise
\begin{align*}
Cv^{1/2n}\geq\Vol^*(E^*\triangle B^n(\lambda^{-1}\tilde x,\om_n^{-1/n}))&\geq\Vol^*(B^n(\lambda^{-1}\tilde x,(1-c_1v^{1/2n^2})\om_n^{-1/n}))\\
&\geq \Vol^*(\B(\lambda^{-1}\tilde x,(1-c_1v^{1/2n^2})\om_n^{-1/n}))\\
&=\lambda^{-n}n\om_n\phi'((1-c_1v^{1/2n^2})\om_n^{-1/n})),
\end{align*}
leading to a contradiction when $v>0$ is small enough, since $\lambda^n\sim v$.
\end{proof}
To complete the proof of $L^\infty$-proximity for the rescaled isoperimetric set $E^*$, we need to prove that the center of the ball $\lambda^{-1}\tilde x$ goes to the origin as $v$ goes to zero. This is possible is we impose the lifting point $o\in M$ to be the barycenter of $E$, that is $o=\bary(E)$.
\begin{lemma}\label{lem:admissible_bary} Let $E\subset M$ be as in Proposition \ref{prop:L1_proximity}. Then, the barycenter $\bary(E)$ is admissible, in the sense that
\[
E\subset\B(\bary(E),\mu v^{1/n}).
\]
\end{lemma}
\begin{proof}In virtue of Proposition \ref{prop:diameter_decay}, there exists $o\in M$ such that $E\subset\B(o,v^{1/n}\mu/2)$. If $p:=\bary(E)\in\B(o,v^{1/n}\mu/2)$ we are done. If this is not the case, then 
\[
g(\exp_p^{-1}(x),\exp^{-1}_p(o))>0,
\]
for all $x\in E$, contradicting Equation \eqref{eq:barycenter_exp}.
\end{proof}
This allows us to chose $o=\bary(E)$ in Proposition \ref{prop:L1_proximity} and Proposition \ref{prop:Linfty_proximity}. We can prove that the associated center of the Euclidean ball $\lambda^{-1}\tilde x$ goes to the origin as $v$ goes to zero.
\begin{proposition}\label{prop:final_Linfty} Let $E\subset M$ be as in Proposition \ref{prop:L1_proximity}. Choosing the normal coordinates pointed at $o=\bary(E)\in M$, we have that there exists $C=C(n,d,\bar v)>0$ so that
\[
\abs{\lambda^{-1}\tilde x}\leq Cv^{1/2n^2}.
\]
In particular, there exists $c_2=c_2(n,d,\bar v)>0$ such that
\begin{equation}\label{eq:bary_at_zero_Linfty}
B^n(0,(1-c_2 v^{1/2n^2})\om_n^{-1/n})\subset E^*\subset B^n(0,(1+c_2 v^{1/2n^2})\om_n^{-1/n}).
\end{equation}
\end{proposition}
\begin{proof}
Recall that by Equation \eqref{eq:barycenter_at_the_origin}, if the barycenter of $E$ is at the origin with respect to the normal coordinates $x=r\varphi$, $(r,\varphi)\in (0,+\infty)\times S^{n-1}$, then
\[
0=\int_Ex\omega_g(r)\,d\Haus^n.
\]
Rescaling, we have that
\[
0=\int_{E^*}x\omega_g(\lambda r)\,d\Haus^n.
\]
Therefore
\begin{align*}
\abs{\lambda^{-1}\tilde x}&=\abs*{\int_{E^*}(\lambda^{-1}\tilde x-x)+x(\omega_g(\lambda r)-1)\,d\Haus^n}\\
&\leq\abs*{\int_{E^*-\lambda^{-1}\tilde x}x\,d\Haus^n}+\int_{E^*}r(\omega_g(\lambda r)-1)\,d\Haus^n.
\end{align*}
Thank to Proposition \ref{prop:Linfty_proximity}, we can estimate the first integral as follows:
\begin{align*}
\abs*{\int_{E^*-\lambda^{-1}\tilde x}x\,d\Haus^n}&\leq\abs*{\int_{S^{n-1}}\int_{\om_n^{-1/n}(1-c_1v^{1/2n^2})}^{\om_n^{-1/n}(1+c_1v^{1/2n^2})}\chi_{E}(r\varphi)r^n\varphi\,dr\,d\varphi}\\
&\leq\PER(\B)\int_{\om_n^{-1/n}(1-c_1v^{1/2n^2})}^{\om_n^{-1/n}(1+c_1v^{1/2n^2})}r^n\,dr\\
&=\frac{\PER(\B)\om_n^{-(n+1)/n}}{n+1}\sum_{k=0}^{n+1}\binom{n+1}{k}(1-(-1)^k)c_1^kv^{k/2n^2}\\
&\leq \frac C2 v^{1/2n^2},
\end{align*}
for some constant $C>0$. Lemma \ref{lem:admissible_bary} implies that $E^*\subset B^n(0,K)$, for some universal $K>0$. Hence, we estimate the second integral as
\begin{align*}
\int_{E^*}r(\omega_g(\lambda r)-1)\,d\Haus^n&=\int_{E^*}r\int_0^{\lambda r}\omega'_g(\tau)\,d\tau\,d\Haus^n\\
&\leq K^2\lambda\omega_g'(\lambda K)\\
&\leq \frac{C}{2}v^{1/n},
\end{align*}
showing that $\abs{\lambda^{-1}\tilde x}\leq Cv^{1/2n^2}$. Equation \eqref{eq:bary_at_zero_Linfty} is a consequence of this and Propositon \ref{prop:Linfty_proximity}.
\end{proof}
After stating the key regularity result, the  proof of Theorem \ref{thm:small_volumes} will be a corollary of Theorem \ref{thm:1} and Proposition \ref{prop:final_Linfty}.
\begin{theorem}\label{thm:regularity}Let $\seq{E_\varepsilon}_{\varepsilon>0}$ be a sequence of sets with finite perimeter in $\R^n$ and $\seq{\F_\varepsilon}_{\varepsilon>0}$ a sequence of functionals of the form
\[
\F_\varepsilon(G):=\int_{\partial^*G}f_\varepsilon(x,\nu)\,d\Haus^{n-1}(x),
\]
where $G$ is a generic set of finite perimeter, $\nu$ its measure theoretic outwards unit normal and $\seq{f_\varepsilon}_{\varepsilon>0}$ a family  of $C^2$-functions, uniformly $\lambda$-elliptic, and with uniformly bounded Hessian in a fixed ball $B(0,2R)$, that is
\[
\sup\Bigl\{\abs{D^2_{\xi\xi}f_\varepsilon(x,\xi)}:(x,\xi)\in B(0,2R)\times S^{n-1}\Bigr\}\leq \Lambda,
\]
for a universal constant $\Lambda>0$. If for all $\varepsilon>0$
\[
B^n(0,(1-\varepsilon)r)\subset E_\varepsilon\subset B^n(0,(1+\varepsilon)r),
\]
and $E_\varepsilon$ is uniformly almost-minimizing  with respect to $\F_\varepsilon$, then there exists $\varepsilon_1>0$ such that for all $0<\varepsilon\leq \varepsilon_1$
\[
\partial E_\varepsilon=\Bigl\{r(1+\rho_\varepsilon(\varphi)):\varphi\in S^{n-1}\Bigr\},
\]
where $\rho_\varepsilon\in C^1(S^{n-1})$ and $\norm{\rho_\varepsilon}_{C^1}\to 0$ as $\varepsilon\to 0$.
\end{theorem}
\begin{proof}See \cite[Theorem 2.2]{figalli2017regularity}.
\end{proof}
We are now ready to prove Theorem \ref{thm:small_volumes}.
\begin{proof}[Proof of Theorem \ref{thm:small_volumes}]Consider a sequence of decreasing volumes $\bar v\geq v_k\to 0$, and let $E_k$ be one isoperimetric region with volume $v_k$ in $M$. Let $E^*_k$ be the rescaling of $E_k$ pointed at $\bary(E_k)$, so that $\VOL(E^*_k)=1$. Then, looking at $E_k^*$ as a sequence of sets with finite perimeter in $\R^n$, we can apply Proposition \ref{prop:final_Linfty} and Theorem \ref{thm:regularity} to infer that there exists $\rho^*_k\in C^1(S^{n-1})$ so that
\[
\partial E^*_k=\{\om_n^{-1/n}(1+\rho^*_k(\varphi)):\varphi\in S^{n-1}\},
\]
and $\norm{\rho^*_k}_{C^1}\to 0$ as $k\to\infty$. Therefore, letting $R_k>0$ be such that $\B(R_k)=v_k$, and $\lambda_k^n=\VOL(E_k)$, we have that
\[
\partial E_k=\{R_k(1+\rho_k(\varphi)):\varphi\in S^{n-1}\},
\]
where $\rho_k:=\frac{\lambda_k\om_n^{-1/n}}{R_k}-1+\frac{\lambda_k\om_n^{-1/n}}{R_k}\rho_k^*\to 0$ in $C^1$ as $k\to+\infty$. Applying Theorem \ref{thm:1}, we conclude the proof.
\end{proof}
\subsection*{Acknowledgments}
The author would like to thank Professors A. Figalli and U. Lang for their guidance and constant support. The author has received funding from the European Research Council under the Grant Agreement No. 721675 “Regularity and Stability in Partial Differential Equations (RSPDE)”.
\printbibliography

@article{Hsiang,
author = {Wu-Yi Hsiang},
title = {{On soap bubbles and isoperimetric regions in noncompact symmetric spaces, I}},
volume = {44},
journal = {Tohoku Mathematical Journal},
number = {2},
publisher = {Tohoku University, Mathematical Institute},
pages = {151 -- 175},
year = {1992},
}

@article{Nardulli,
   title={Sharp Isoperimetric Inequalities for Small Volumes in Complete Noncompact Riemannian Manifolds of Bounded Geometry Involving the Scalar Curvature},
   volume={2020},
   ISSN={1687-0247},
   number={15},
   journal={International Mathematics Research Notices},
   publisher={Oxford University Press (OUP)},
   author={Nardulli, Stefano and Osorio Acevedo, Luis Eduardo},
   year={2018},
   month={Jun},
   pages={4667–4720}
}

@article{FigalliMaggi,
author = {Figalli, Alessio and Maggi, F. and Pratelli, Aldo},
year = {2010},
month = {10},
pages = {167-211},
title = {A mass transportation approach to quantitative isoperimetric inequalities},
volume = {182},
journal = {Inventiones Mathematicae},
}

@article{Duzaar,
author = {Duzaar, Frank and Steffen, Klaus},
year = {2000},
month = {05},
pages = {},
title = {Optimal Interior and Boundary Regularity for Almost Minimizers to Elliptic Variational Integrals},
volume = {2002},
journal = {Journal fur die Reine und Angewandte Mathematik},
}

@book{Helgason,
	address = {Providence, Rhode Island},
	author = {Helgason, Sigurdur},
	booktitle = {Differential geometry, Lie groups, and symmetric spaces},
	language = {eng},
	publisher = {American Mathematical Society},
	series = {Graduate studies in mathematics, volume 34},
	title = {Differential geometry, Lie groups, and symmetric spaces},
	year = {2001 - 1978}}

@book{Eberlein,
	address = {Chicago},
	author = {Eberlein, Patrick},
	language = {eng},
	publisher = {University of Chicago Press},
	series = {Chicago lectures in mathematics series},
	title = {Geometry of nonpositively curved manifolds},
	year = {1996}}

@article{Steiner,
author = {J. Steiner},
title = {Einfache Beweise der isoperimetrischen Hauptsätze.: },
journal = {Journal für die reine und angewandte Mathematik (Crelles Journal)},
number = {18},
volume = {1838},
year = {1838},
pages = {281--296}
}

@article{Schwarz,
author = {Schwarz, H. A.},
journal = {Nachrichten von der Königl. Gesellschaft der Wissenschaften und der Georg-Augusts-Universität zu Göttingen},
pages = {1-13},
title = {Beweis des Satzes, dass die Kugel kleinere Oberfläche besitzt, als jeder andere Körper gleichen Volumens},
volume = {1884},
year = {1884},
}

@article{BCE,
author = {Barbosa, J. and Carmo, Manfredo and Eschenburg, Jost},
year = {1988},
month = {03},
pages = {123-138},
title = {Stability of Hypersurfaces of Constant Mean Curvature in Riemannian Manifolds},
volume = {197},
journal = {Mathematische Zeitschrift},
}

@article{RT,
  title={Stability of Geodesic Spheres},
  author={H. J. Rivertz and P. Tomter},
  journal = {Geometry and topology of submanifolds},
  pages = {320 -- 324},
  year={1994}
}

@article{BB,
author = {Lionel Bérard Bergery and Jean-Pierre Bourguignon},
title = {{Laplacians and Riemannian submersions with totally geodesic fibres}},
volume = {26},
journal = {Illinois Journal of Mathematics},
number = {2},
publisher = {Duke University Press},
pages = {181 -- 200},
year = {1982},
}

@article{Fuglede,
 author = {Bent Fuglede},
 journal = {Transactions of the American Mathematical Society},
 number = {2},
 pages = {619--638},
 publisher = {American Mathematical Society},
 title = {Stability in the Isoperimetric Problem for Convex or Nearly Spherical Domains in Rn},
 volume = {314},
 year = {1989}
}

@article{TOM,
author = {P. Tomter},
year = {1993},
pages = {485–495},
title = {Constant mean curvature surfaces in the Heisenberg group.},
journal = {Proc. Sympos. Pure Math.},
}

@article{NARDU,
author = {Nardulli, Stefano},
year = {2009},
month = {09},
pages = {111-131},
title = {The isoperimetric profile of a smooth Riemannian manifold for small volumes},
volume = {36},
journal = {Annals of Global Analysis and Geometry},
}

@article{Fusco,
author = {Fusco, Nicola and Julin, Vesa},
year = {2011},
month = {11},
pages = {},
title = {A strong form of the Quantitative Isoperimetric inequality},
volume = {50},
journal = {Calculus of Variations},
}

@book{Maggi_perimeter, place={Cambridge}, series={Cambridge Studies in Advanced Mathematics}, title={Sets of Finite Perimeter and Geometric Variational Problems: An Introduction to Geometric Measure Theory}, publisher={Cambridge University Press}, author={Maggi, Francesco}, year={2012}, collection={Cambridge Studies in Advanced Mathematics}}

@book{ambro,
	address = {Oxford},
	author = {Ambrosio, Luigi},
	language = {eng},
	publisher = {Clarendon Press},
	series = {Oxford science publications},
	title = {Functions of bounded variation and free discontinuity problems},
	year = {2000}}

@article{fusco2008sharp,
  title={The sharp quantitative isoperimetric inequality},
  author={Fusco, Nicola and Maggi, Francesco and Pratelli, Aldo},
  journal={Annals of mathematics},
  pages={941--980},
  year={2008},
  publisher={JSTOR}
}

@misc{bettiol2022laplace,
      title={Full Laplace spectrum of distance spheres in symmetric spaces of rank one}, 
      author={Renato G. Bettiol and Emilio A. Lauret and Paolo Piccione},
      year={2022},
      eprint={2012.02349},
      archivePrefix={arXiv},
      primaryClass={math.DG}
}

@misc{Antonelli,
  
  url = {https://arxiv.org/abs/2201.04916},
  
  author = {Antonelli, Gioacchino and Pasqualetto, Enrico and Pozzetta, Marco and Semola, Daniele},
  
  title = {Sharp isoperimetric comparison on non collapsed spaces with lower Ricci bounds},
  
  publisher = {arXiv},
  
  year = {2022},

}

@book{Hebey,
language = {eng},
publisher = {Courant Institute of Mathematical Sciences},
author = {Hebey, Emmanuel},
address = {New York, New York ;},
booktitle = {Nonlinear analysis on manifolds : Sobolev spaces and inequalities},
isbn = {1-4704-1759-6},
series = {Courant Lecture Notes in Mathematics ; Volume 5},
title = {Nonlinear analysis on manifolds  : Sobolev spaces and inequalities },
year = {2000 - 1999},
}

@book{Gromov,
publisher = {Boston, MA : Birkhäuser Boston},
author = {Gromov, Mikhael and Katz, M. and Pansu, P. and Semmes, S. and LaFontaine, Jacques},
title = {Metric Structures for Riemannian and Non-Riemannian Spaces},
year = {2014},
}

@article{FM1,
  title={Clusters minimizing area plus length of singular curves},
  author={Morgan, Frank},
  journal={Mathematische Annalen},
  volume={299},
  number={1},
  pages={697--714},
  year={1994},
  publisher={Springer}
}

@article{GR,
  title={Existence of isoperimetric regions in contact sub-Riemannian manifolds},
  author={Galli, Matteo and Ritor{\'e}, Manuel},
  journal={Journal of Mathematical Analysis and Applications},
  volume={397},
  number={2},
  pages={697--714},
  year={2013},
  publisher={Elsevier}
}

@article{cicalese2012selection,
  title={A selection principle for the sharp quantitative isoperimetric inequality},
  author={Cicalese, Marco and Leonardi, Gian Paolo},
  journal={Archive for Rational Mechanics and Analysis},
  volume={206},
  number={2},
  pages={617--643},
  year={2012},
  publisher={Springer}
}

@article{figalli2017regularity,
  title={Regularity of codimension-1 minimizing currents under minimal assumptions on the integrand},
  author={Figalli, Alessio},
  journal={Journal of Differential Geometry},
  volume={106},
  number={3},
  pages={371--391},
  year={2017},
  publisher={Lehigh University}
}
\end{document}